\documentclass[11pt,a4paper,headinclude,footinclude,fleqn,reqno]{amsart}                 
\usepackage[T1]{fontenc}                   
\usepackage[utf8]{inputenc}                 
\usepackage[english]{babel}       
\usepackage{graphicx}                      %
\usepackage[font=small]{quoting}            %
\usepackage{caption}  
\usepackage[colorlinks=true,linkcolor=blue]{hyperref}
                  
\usepackage[top=.8in,bottom=1.2in,left=1.5in,right=1.5in]{geometry}
\usepackage{verbatim}
\usepackage{color}
\usepackage{picture}
\usepackage{amsmath,amsthm,amsfonts,amssymb}
\usepackage{comment}

\newtheorem{thm}{Theorem}[section] 
\newtheorem{lem}[thm]{Lemma} 
\newtheorem{cor}[thm]{Corollary} 
\newtheorem{prop}[thm]{Proposition} 
 
\newtheorem{defn}{Definition} 
\theoremstyle{definition} 
\newtheorem{rem}[thm]{Remark} 
\theoremstyle{remark}

\def\O{\Omega}

\def\S{\Sigma} 
\def\n{\nabla}

\def\p{\partial}

\def\a{\alpha}
\def\b{\beta}
\def\n{\nabla}

\def\O{\Omega}
\def\p{\partial}

\def\a{\alpha}
\def\b{\beta}
\def\g{\gamma}
\def\d{\delta}
\def\k{\kappa}

\def\s{\sigma}

\def\n{\nabla}
\def\<{\langle}
\def\>{\rangle}
\def\div{{\rm div}}

\def\n{\nabla}

\def\O{\Omega}
\def\p{\partial}

\def\a{\alpha}
\def\b{\beta}
\def\g{\gamma}
\def\d{\delta}

\def\s{\sigma}

\def\Rn{\mathbb R^n}

\def\de{\partial}
\def\eps{\varepsilon}

\def\R{\mathbb{R}}

\DeclareMathOperator{\dive}{div}

\DeclareMathOperator{\supp}{supp}

{\left\{\begin{array}{@{}l@{}}}{\end{array}\right.}
\patchcmd{\abstract}{\scshape\abstractname}{\textbf{\abstractname}}{}{}
\makeatletter 
\def\@makefnmark{} 
\makeatother

\begin{document}
\title{Motion of level sets by inverse anisotropic mean curvature}
\author{Francesco Della Pietra}
\address{Universit\`a degli studi di Napoli Federico II, Dipartimento di Matematica e Applicazioni ``R. Caccioppoli'', Via Cintia, Monte S. Angelo - 80126 Napoli, Italia.}
 \email{f.dellapietra@unina.it}
\author{Nunzia Gavitone}
\address{Universit\`a degli studi di Napoli Federico II, Dipartimento di Matematica e Applicazioni ``R. Caccioppoli'', Via Cintia, Monte S. Angelo - 80126 Napoli, Italia.}
 \email{nunzia.gavitone@unina.it}
\author{Chao Xia}
\address{School of Mathematical Sciences, Xiamen University, 361005, Xiamen, China}
 \email{chaoxia@xmu.edu.cn}

\maketitle
\begin{abstract}
In this paper we consider the weak formulation of the inverse  anisotropic mean curvature flow, in the spirit of Huisken-Ilmanen \cite{HI1}. By using approximation method involving Finsler-$p$-Laplacian, we prove the existence and uniqueness of weak solutions.
\end{abstract}



\section{Introduction}

Let $F\in C^\infty(\mathbb{R}^n\setminus\{0\})$ be a Minkowski norm in $\mathbb{R}^{n}$, i.e.,
\begin{itemize}
\item[(i)] $F$ is a norm in $\mathbb{R}^{n}$, i.e.,  $F$ is a convex, even, $1$-homogeneous function satisfying $F(\xi)>0$ when $\xi\ne 0$;
\item[(ii)] $F$ satisfies a uniformly elliptic condition: $D^2 (\frac12 F^2)$ is positive definite in $\mathbb{R}^{n}\setminus \{0\}$.
\end{itemize} 

Let $X(\cdot,t)\colon M\times[0,T)\to \R^{n}$ be a family of smooth embeddings from a closed manifold $M$  in $\R^{n}$ satisfying the evolution equation
\begin{equation}
\label{iamcf}
\frac{\de}{\de t}X(x,t)=\frac{1}{H_{F}(x,t)}\nu_{F}(x,t),
\end{equation}
where $H_{F}(x,t)>0$ is the anisotropic mean curvature function of the hypersurface $N_{t}=X(M,t)$ and $\nu_{F}(x,t)$ is the unit anisotropic outer normal. 

When $F$ is the Euclidean norm, $\nu_F$ and $H_F$ reduce to the unit outer normal and the mean curvature respectively, and in turn \eqref{iamcf} reduces to the classical inverse mean curvature flow (IMCF). 
When $F$ is a general Minkowski norm,  \eqref{iamcf} is so-called inverse anisotropic mean curvature flow (IAMCF).

Gerhardt \cite{Ge} and Urbas \cite{Ur} proved that the classical IMCF which initiated from a star-shaped and strictly mean convex hypersurface exists for all time and converge to a round sphere after rescaling. For general initial data, the IMCF may develop singularity. Huisken-Ilmanen \cite{HI1} has developed a theory of weak solutions for the IMCF of hypersurfaces in Riemannian manifolds by its level-set formulation and applied it to show the validity of the Riemannian Penrose inequality.

For the anisotropic counterpart, recently the third author \cite{x} has studied the IAMCF which initiated from a star-shaped and strictly $F$-mean convex hypersurface and proved the  long time existence and convergence result analogous to Gerhardt and Urbas' result. The aim of this paper is to study Huisken-Ilmanen type weak solutions for the IAMCF by its level-set formulation.

Suppose the evolving hypersurfaces $N_{t}$ are given by level sets of a function $u\colon \R^{n}\to \R$, that is
\[
E_{t}= \{x\in \R^{n}: u(x)<t\}, \quad N_t=\p E_t.
\] 
If $u$ is smooth and $\nabla u\ne 0$, then \eqref{iamcf} is equivalent to the degenerate elliptic equation
\begin{equation}
\label{prob1}
\dive\left(F_{\xi}(\nabla u) \right) = F(\nabla u).
\end{equation}
See Section 2.
When $F$ is Euclidean, it is clear \eqref{prob1} reduces to
\begin{equation}
\label{prob2}
\dive\left(\frac{\nabla u}{|\nabla u|}\right) = |\nabla u|.
\end{equation}

As Huisken-Ilmanen \cite{HI1}, we define a weak solution of \eqref{prob1} by the following minimization principle.
\begin{defn}
\label{defweak}
Let $\O\subset \R^n$ be an open set.
A function $u\in C^{0,1}_{\textrm{loc}}(\Omega)$ is called a weak solution 
to \eqref{prob1} if 
\begin{equation}
\label{mineq}
J_{F,u}(u)\le J_{F,u}(\varphi)
\end{equation}
for every precompact set $K\subset \Omega$ and for every test function $\varphi\in C^{0,1}_{\textrm{loc}}(\Omega)$ 
 with $\varphi=u$ in $\Omega\setminus K$, and where 
\begin{equation}\label{JFu}
J_{F,u}(\varphi):=\int_{K} \big[F(\nabla \varphi)+\varphi F(\nabla u)\big]dx.
\end{equation}

Moreover, $u$ is a proper solution if in addition
\[
\lim_{|x|\to +\infty} u(x)=+\infty.
\]
\end{defn}

Our main result of this paper is the following existence result. 

\begin{thm}
\label{thm-existence}
Let $\O\subset \R^n$ be an open set with smooth boundary such that $\O^c=\R^n\setminus \O$ is bounded.
There exists a unique proper weak solution $u\in C^{0,1}_{\textrm{loc}}(\overline \Omega)$ of \eqref{prob1}, in the sense of Definition \ref{defweak}, such that $u=0$ on $\de \O$. Moreover, $u$ satisfies
\begin{align}
\label{boundIMCF}
F(\nabla u(x)) \le \sup_{\de \Omega}H^+_{F}, \quad  x \in \overline\Omega,\\
F(\nabla u(x)) \le H^{+}_{F}(x),\quad  x \in \de\Omega,
\end{align}
where $H_F^+(x)=\max\{H_F(x),0\}$ and $H_F$ is the anisotropic mean curvature of $\de \Omega$.
\end{thm}

Huisken-Ilmanen's approach in the classical IMCF case to prove the existence is  studying an  approximate equation of \eqref{prob2}, known as elliptic regularization. One of the key feature of this elliptic regularization is that it corresponds to a family of translating graphs which solves the IMCF in $\R^{n}\times \R$.  
It seems such elliptic regularization is not available in the anisotropic case.
Later, Moser \cite{m} found another  approximate equation of \eqref{prob2} involving the $p$-Laplacian. It turns out that this approximate equation is also effective to prove the existence of weak solutions for IAMCF.

Inspired by Moser's approach, we consider the approximate equation of \eqref{prob1} involving the Finsler-$p$-Laplacian, that is,
\begin{equation}
\label{p_prob}
\left\{
\begin{array}{ll}
\dive\left( F^{p-1}(\nabla u)F_{\xi}(\nabla u) \right) = F(\nabla u)^{p} & \text{in }\Omega, \\[.2cm]
u=0 &\text{in }\Omega^{c},\\
u\to \infty&\text{as }x\to \infty.
\end{array}
\right.
\end{equation}
We have the following 
\begin{thm} \label{theoremapprox}Let $\O\subset \R^n$ be an open set with smooth boundary such that $\O^c=\R^n\setminus \O$ is bounded.
For every $p>1$, there exists a unique solution $u_p\in C^{1,\alpha}_{loc}(\overline{\O})$. Moreover,
for every $\eps >0$, there exists $p_0=p_0(\eps)>1$ such that  if $u_p\in C^{1,\alpha}_{loc}(\overline{\O})$ is  the solution to \eqref{p_prob} for   $1<p\le p_0$, then 
\begin{align}
&F(\nabla u_p(x)) \le \sup_{\de \Omega}H^+_{F}+\eps, \quad  x \in \overline\Omega\label{estcurv1}\\
&F(\nabla u_{p}(x)) \le H^+_{F}(x)+\eps, \quad x \in \de \Omega.\label{estcurv2}
\end{align}
\end{thm}

Theorem \ref{thm-existence} follows from Theorem \ref{theoremapprox} by approximation. 

The rest of this paper is organized as follows. In Section 2, we recall some fundamentals on anisotropic functions and anisotropic mean curvature. In Section 3, we study Huisken-Ilmanen type weak formulation of IAMCF and its properties. In Section 4, we study the approximate equation involving the Finsler-$p$-Laplacian and  show the gradient estimate and the existence of weak solution of IAMCF.

\

\section{Notation and preliminaries}

\subsection{Minkowski norm and Wulff shape}\

Let $F$ be a Minkowski norm on $\R^n$.
The polar function $F^o\colon \Rn \rightarrow [0,+\infty[$ 
of $F$, defined as
\begin{equation*}
F^o(x)=\sup_{\xi\ne 0} \frac{\langle \xi, x\rangle}{F(\xi)},
\end{equation*}
is again a Minkowski norm on $\R^n$.
Furthermore, 
\begin{equation*}
F(\xi)=\sup_{x\ne 0} \frac{\langle \xi, x\rangle}{F^o(x)}.
\end{equation*}
Denote 
\[
\mathcal W = \{  x\in  \Rn \colon F^o(x)< 1 \}. 
\]
This is the so-called Wulff shape centered at the origin. 
More generally, we denote by $\mathcal W_r(x_0)$
the set $r\mathcal W+x_0$, that is the Wulff shape centered at $x_0$
with radius $r$ and  $\mathcal W_r=\mathcal W_r(0)$.


The following properties of $F$ and $F^o$ hold true:
for any  $x, \xi \in \Rn\setminus \{0\}$,
\begin{gather*}
\label{prima}
 \langle F_\xi(\xi) , \xi \rangle= F(\xi), \quad  \langle F_x^{o} (x), x \rangle
= F^{o}(x)
 \\
 \label{seconda} F(  F_x^o(x))=F^o(  F_\xi(\xi))=1,\\
\label{terza} 
F^o(x)   F_\xi( F_x^o(x) ) =x, \quad F(\xi) 
 F_x^o\left(  F_\xi(\xi) \right) = \xi.\end{gather*}
See e.g. \cite{X2}, Chapter 2.

\subsection{Anisotropic mean curvature and anisotropic area functional}\

Let $N$ be a smooth closed hypersurface in $\Rn$ and $\nu$ be the unit Euclidean outer normal of $N$.
The anisotropic outer normal of $N$ is defined by \[
\nu_F=F_\xi(\nu).\]
The anisotropic mean curvature of $N$ is defined by 
\[H_F=\dive(\nu_F).
\]
Here $\dive$ is the divergence operator with respect to the Euclidean metric.
In this paper we are interested in the case  when $N$ is given by a level set of a smooth function $u$, namely, \[
N=N_{t}=\de E_t,\hbox{ where }E_t=\{x\in \Rn: u(x)<t\}.\]
When $\nabla u\neq 0$, it is clear that $\nu=\frac{\nabla u}{|\nabla u|}$ and \begin{eqnarray}\label{levelset}
\nu_F= F_\xi(\nabla u), \quad H_F= \dive(F_\xi(\nabla u)).
\end{eqnarray}
If $N_t$ satisfies the IAMCF, we see that $u(x(t))=t$ and by taking derivative about $t$,
 we get \[\left\<\nabla u, \frac{1}{H_F}\nu_F\right\>=1.\] By virtue of \eqref{levelset}, we arrive at \eqref{prob1}.
 
 The anisotropic area functional of $N$ is defined as
\[
	\sigma_{F}(N)=\int_{N} F(\nu)\,d \mathcal{H}^{n-1},
\]
where $\mathcal{H}^{n-1}$ is the $(n-1)$-dimensional Hausdorff measure.

It is well-known that a variational characterization of $H_F$ is given by the first variational formula of $\sigma_{F}$,  see for instance \cite{Re2, bpr, x}.
More precisely, we have
\begin{prop}[Reilly \cite{Re2}, Bellettini-Novaga-Riey \cite{bpr}]\label{HL}\

Let $N$ be a smooth closed hypersurface given by an embedding $X_0: M\to \Rn$. Let $N_s$ be a variation of $N$ given by $X(\cdot,s): M\to \Rn$, $s\in (- \eps, \eps),$ whose variational vector field $\frac{\p}{\p s}|_{s=0}X(\cdot,s) =V$. Then 
\begin{eqnarray}\label{var0}
\left.\frac{d}{d s}\right |_{s=0} \sigma_{F}(N_{s})=
\int_{N}  \dive_{F,N}(V) F(\nu) d \mathcal{H}^{n-1}= 
\int_N H_F(X_0) \left\<V, \nu\right\> d \mathcal{H}^{n-1},
\end{eqnarray}
where 
\[
\dive_{F,N}(V) 
:= {\rm div} V- \left\<\nabla_{\nu_F} V, \frac{\nu}{F(\nu)}\right\>.
\]
\end{prop}

\begin{proof} We refer to \cite{x} for the proof of the second equality.  For completeness, we prove the first equality here. We denote $\nu_s$ and $d\sigma_s$ be the unit outer normal and the area element of $N_s$ respectively.
It is well-known that 
\[
\frac{d}{ds}\Big|_{s=0}\nu_s=-\<\nu, \nabla_{e_i}V\>e_i, 
\]
where $\{e_i\}$ is an orthonormal basis of $TN$.
\[
\frac{d}{ds}\Big|_{s=0}d\sigma_s=\div_N(V) d\sigma.
\]
Thus
\begin{align*}
\frac{d}{ds}\Big|_{s=0} F(\nu_s) d\sigma_s&=\<F_\xi(\nu), -\<\nu, \nabla_{e_i}V\>e_i\>d\sigma+ F(\nu)\dive_N(V) d\sigma
\\&=\left[-\<\nu, \nabla_{e_i}V\>\<F_\xi(\nu), e_i\>+(\dive V- \<\nabla_\nu V, \nu\>)F(\nu)\right]  d\sigma
\\&= \left[\div V -\left\<\nabla_{\nu_F} V, \frac{\nu}{F(\nu)}\right\> \right]F(\nu)d\sigma
\end{align*}
The last line follows from \[\nu_F=F_{\xi}(\nu)=F(\nu)\nu+\<F_\xi(\nu), e_i\>e_i.\]
\end{proof}

\

\section{IAMCF: a variational formulation}

In this section, we review the weak formulation of IAMCF developed by Huisken-Ilmanen \cite{HI1} by using a minimizing principle. We follow closely Huisken-Ilmanen's strategy in \cite{HI1}, Section 1.

\subsection{Weak formulation of IAMCF} 
Recall (Definition \ref{defweak}) that $u$ is called a weak solution (subsolution, supersolution resp.) of \eqref{prob1} in $\O$ if $u\in C^{0,1}_{\textrm{loc}}(\Omega)$ and 
$J_{F,u}(u)\le J_{F,u}(\varphi)$
for every precompact set $K\subset \Omega$ and for every test function $\varphi\in C^{0,1}_{\textrm{loc}}(\Omega)$ ($\varphi\le u$, $\varphi \ge u$ resp.) with $\varphi=u$ in $\Omega\setminus K$, where $J_{F,u}$ is defined in \eqref{JFu}.

The fact that
\[
J_{F,u}(\min\{\varphi,u\})+J_{F,u}(\max\{\varphi,u\})=J_{F,u}(\varphi)+J_{F,u}(u)
\]
whenever $\{u\neq \varphi\}$ is precompact implies $u$ is a solution if and only if it is both a weak supersolution and a weak subsolution.

There is an equivalent weak formulation by set functional. For $K\subseteq \Omega$ and $u\in C^{0,1}_{\textrm{loc}}(\Omega)$, define
\begin{equation*}
J_{F, u}(G)=J_{F,u}^{K}(G):=\int_{\de^{*}G\cap K} F(\nu)d\mathcal{H}^{n-1}-\int_{G\cap K}F(\nabla u)dx,
\end{equation*}
for a set $G$ of locally finite perimeter, and $\de^{*}G$ denotes the reduced boundary of $G$. 
\begin{defn}
\label{setdef}
We say that $E$ minimizes $J_{F,u}$ in a set $A$  (minimizes on the outside, minimizes on the inside, resp.) if 
\begin{equation*}
J_{F,u}(E)\le J_{F,u}(G)
\end{equation*}
for any $G$  such that $E\Delta G\subset\subset A$ ($G\supseteq E$, $G\subseteq E$ resp.) and any compact set $K$ containing $E\Delta G$. 
Here $E\Delta G=(E\setminus G)\cup (G\setminus E)$.
\end{defn}
The fact that
\[
J_{F,u}(E\cap G)+J_{F,u}(E\cup G)\le J_{F,u}(E)+J_{F,u}(G)
\]
whenever $E\Delta G$ is precompact guarantees that $E$ minimizes $J_{F,u}$ in $\O$ if and only if $E$ is minimizes $J_{F,u}$ both on the inside and on the outside in $\O$.

The Definitions \ref{defweak} and \ref{setdef} are equivalent in the following sense.
\begin{prop}
\label{equiv}
Let $\Omega$ be an open set and $u\in C^{0,1}_{\textrm{loc}}(\Omega)$, then $u$ is a weak solution of \eqref{prob1} in $\Omega$ if and only if for each $t$,
$E_{t}=\{x\in\Omega \colon u<t\}$
minimizes $J_{F,u}$ in $\Omega$.
\end{prop}
\begin{proof}
By the co-area formula, we have for a choice of $a<b$ such that $a<u<b$ and $a<\varphi<b$ in $K$, that
\begin{align}
\notag J_{F,u}(\varphi)&=\int_{K}(F(\nabla \varphi) +\varphi F(\nabla u))
dx=\\ &=\int_{a}^{b} dt\int_{(\de\{\varphi<t\})\cap K}\left(F\left(\frac{\nabla \varphi}{|\nabla \varphi|}\right)+\frac{t}{|\nabla \varphi|} F(\nabla u)\right)d\sigma\\
\notag &= \int_{a}^{b} \int_{(\de\{\varphi<t\})\cap K}F(\nu)\,d\sigma -\int_{K}\int_{a}^{b}\chi_{\{\varphi<t\}}F(\nabla u)d\sigma + \\
&\notag\qquad\qquad\qquad\qquad\qquad\qquad\qquad\qquad\qquad+b\int_{K}F(\nabla u)\,dx\\
\notag &= \int_{a}^{b} J^{K}_{u}(\{\varphi<t\}) + b\int_{K}F(\nabla u)\,dx.
\end{align}
Then, if for any $t$, $E_{t}$ is a minimizer of the set functional $J_{F,u}$, then
\[
J_{F,u}(\varphi) \ge J_{F,u}(u),
\]
that gives the minimality of $u$. 

The viceversa can be proved  exactly  as in the proof of Lemma 1.1 in \cite{HI1}. 
\end{proof}

Next we study the weak formulation of IAMCF with initial condition.
\begin{defn}\label{weak2def} Let $E_0$ be an open set with smooth boundary. Let $\{E_t\}_{t>0}$ be a nested family of open sets in $\Rn$.
\begin{itemize}
\item[(i)] 
 $u$ is called a weak solution of \eqref{prob1} with initial condition $E_0$ if $u\in C^{0,1}_{loc}(\Rn)$, $E_0=\{u<0\}$ and $u$ is a weak solution of \eqref{prob1} in $\Rn\setminus \overline{E_0}$.
\item[(ii)] Define $u$ by the characterization $E_t=\{u<t\}$.
$\{E_t\}_{t>0}$ is called a weak solution of \eqref{iamcf} with initial condition $E_0$ if $u\in C^{0,1}_{loc}(\Rn)$ and $E_t$ minimizes $J_{F,u}$ in $\Rn\setminus E_0$ for each $t>0$.
\end{itemize}
\end{defn}

From Proposition \ref{equiv}, it is easy to see the above two definitions are also equivalent.
\begin{prop}  $u$ is a weak solution of \eqref{prob1} with initial condition $E_0$ if and only if $\{u<t\}_{t>0}$ is a weak solution of \eqref{iamcf} with initial condition $E_0$.
\end{prop}

The weak solution is unique.
\begin{prop}[Uniqueness of the weak solutions]
\begin{enumerate}
\item Let $u$ and $v$ be weak solutions to \eqref{prob1} in $\Omega$ in the sense of Definition \ref{defweak}, and $\{v>u\}\subset\subset\Omega$. Then $v\le u$ in $\Omega$;
\item if $\{E_{t}\}_{t>0}$ and $\{F_{t}\}_{t>0}$ solve \eqref{iamcf} in the sense of Definition \ref{weak2def}, with initial data $E_{0},F_{0}$ respectively, and $E_{0}\subseteq F_{0}$, then $E_{t}\subseteq F_{t}$ as long as $E_{t}$ is precompact. In particular, for a given $E_{0}$ there exists at most one solution $\{E_{t}\}_{t>0}$ of \eqref{iamcf} such that $E_{t}$ is precompact.
\end{enumerate}
\end{prop}
\begin{proof}
See Huisken-Ilmanen \cite{HI1}, page 377-378.
\end{proof}

\subsection{Properties of weak IAMCF}
\begin{defn} Let $\O$ be an open set.
\begin{itemize}
\item[(i)]  A set $E$ is called an $F$-minimizing hull in $\O$ if 
\begin{equation}\label{f-mini}
\sigma_F(\de^*E\cap K)\le \sigma_F(\de^*G\cap K)
\end{equation}
for any $G$ containing $E$ such that $G\setminus E\subset\subset \O$ and any compact set $K$ containing $G\setminus E$.
\item[(ii)] A set $E$ is called a strictly $F$-minimizing hull in $\O$ if it is an $F$-minimizing hull in $\O$ and equality holds in \eqref{f-mini} if and only if 
\[
G\cap \O= E\cap \O \hbox{ a.e. }
\]
\item[(iii)] Given a measurable set $E$, the set $E'$ is defined to be the intersection of all the  strictly $F$-minimizing hulls in $\O$ that contain $E$.
\end{itemize}
\end{defn}
One has the following properties for weak solutions of IAMCF. 

\begin{prop} Let $u$ be a weak solution of \eqref{prob1} with initial condition $E_0$. Set 
\[
E_t=\{u<t\}, \quad E_t^+={\rm int}\{u\le t\}.
\]
Then
\begin{itemize}
\item[(i)]  For $t>0$, $E_t$ is an $F$-minimizing hull in $\Rn$;
\item[(ii)] Fot $t\ge 0$, $E_t^+$ is a strictly $F$-minimizing hull in $\Rn$ and $E_t'=E_t^+$ if it is precompact;
\item[(iii)] For $t>0$, $\sigma_F(\de E_t)= \sigma_F(\de E_t^+)$ provided $E_t^+$ is precompact. This holds true for $t=0$ if $E_0$ is an $F$-minimizing hull. 
\item[(iv)] $\s_F(E_t)=e^t\s_F(E_0)$ provided $E_0$ is an $F$-minimizing hull.
\end{itemize}
\end{prop}
\begin{proof}
See Huisken-Ilmanen \cite{HI1}, page 372-373.
\end{proof}






Analog with the classical case, we define the weak anisotropic mean curvature by the first variational formula, Proposition \ref{HL}.

\begin{defn}
Let $N\subset \Rn$  be a hypersurface of $C^{1} $ or $C^{1}$ with a small singular set and locally finite Hausdorff measure. A locally integrable function $H_{F}$ on $N$ is called weak anisotropic mean curvature provided it satisfies the second equality in \eqref{var0} for every $V \in C_{c}^{\infty}(\Rn)$.
\end{defn}

For smooth IAMCF given by $\{u=t\}$, one sees $H_F=F(\nabla u)$. We show next  weak solutions of \eqref{prob1} still have this property.
\begin{prop}
Let $u$ be a weak solution of \eqref{prob1} with initial condition $E_0$ and let $N_t=\de E_t=\de \{u<t\}$. Then for a.e. $t$, the weak anisotropic mean curvature $H_F$ of $N_t$ satisfies
\[
H_F= F(\nabla u)\hbox{ a.e. }x\in N_t.
\]
\end{prop}

\begin{proof}Let $V \in C_{c}^{\infty}(\Rn)$ and $\Phi^s: \Rn\to \Rn, s\in (-\eps, \eps),$ be the flow of diffeomorphisms generated by $V$ and $\Phi^0=Id$.
Let $W$ be any precompact open set containing ${\rm supp}(V)$. 

Because $u$ be a weak solution of \eqref{prob1} in $\Rn\setminus \overline{E_0}$, we see
$J_{F,u}(u\circ \Phi^s)\le J_{F,u}(u).$
Thus $\frac{d}{ds}\Big|_{s=0}J_{F,u}(u\circ \Phi^s)=0$. Next we derive $\frac{d}{ds}\Big|_{s=0}J_{F,u}(u\circ \Phi^s)$.

First, we assume $u$ is smooth. Then
\begin{align*}
&\frac{d}{ds}\Big|_{s=0}\int_K F(\nabla (u\circ\Phi^s)) dx\\&=
\int_W F_{\xi_i}(\nabla u) \nabla_i (\frac{d}{ds}\Big|_{s=0}(u\circ\Phi^s))dx\\
&=\int_W F_{\xi_i}(\nabla u)\nabla_i (\nabla_j u  V^j)dx\\
&=\int_W F_{\xi_i}(\nabla u) \nabla^2_{ji} u   V^j +F_{\xi_i}(\nabla u) \nabla_j u  \nabla_i  V^j dx
\\&=\int_W -F_{\xi_i\xi_k}(\nabla u) \nabla^2_{kj} u \nabla_i u   V^j -F_{\xi_i}(\nabla u)  \nabla_i u  \nabla_j V^j +F_{\xi_i}(\nabla u) \nabla_j u  \nabla_i  V^j dx
\\&=\int_W -F(\nabla u)\dive V+ \<\nabla u,  \nabla_{F_{\xi}(\nabla u)}  V\> dx.
\end{align*}
By co-area formula, 
\begin{align*}
&\int_W -F(\nabla u)\dive V+\<\nabla u,  \nabla_{F_{\xi}(\nabla u)}  V\> dx
\\&=\int_{-\infty}^\infty \int_{N_t\cap W} -\dive V F(\nu)+ \<\nu, \nabla_{\nu_F} V\> d\sigma_t dt
\\&=\int_{-\infty}^\infty \int_{N_t\cap W} -\dive_{F, N}V  F(\nu)d\sigma_t dt.
\end{align*}
Thus \begin{align}\label{firstvariation}
\frac{d}{ds}\Big|_{s=0}\int_K F(\nabla (u\circ\Phi^s)) dx=\int_{-\infty}^\infty \int_{N_t\cap W} -\dive_{F, N}V F(\nu)d\sigma_t dt.
\end{align}
By an approximation argument, we see that the formula \eqref{firstvariation} is still true for $u$ only locally Lipschitz.

On the other hand, it is easy to see
\begin{align}\label{firstvariation2}
\frac{d}{ds}\Big|_{s=0}\int_K (u\circ\Phi^s) F(\nabla u) dx&=\int_W \<\nabla u, V\>F(\nabla u) dx\nonumber
\\&= \int_{-\infty}^\infty \int_{N_t\cap W}  \<\nu, V\>F(\nabla u) d\sigma_t dt.
\end{align}
Combining \eqref{firstvariation} and \eqref{firstvariation2}, we get
\begin{align}\label{firstvariation3}
0&=\frac{d}{ds}\Big|_{s=0}J_{F,u}(u\circ \Phi^s)\nonumber\\&= \int_{-\infty}^\infty \int_{N_t\cap W} -\dive_{F, N}V F(\nu)+  \<\nu, V\>F(\nabla u) d\sigma_t dt.
\end{align}
Finally, by the definition of the weak anisotropic mean curvature, we conclude from \eqref{firstvariation3} that 
\[
H_F= F(\nabla u)\hbox{ a.e. }x\in N_t \hbox{ a.e. }t.
\]
\end{proof}


\section{Existence of solutions and gradient estimates}

For any $p>1$, we will consider the following auxiliary problem
\begin{equation}
\label{p_change}
\left\{
\begin{array}{ll}
\dive\left( F^{p-1}(\nabla v)F_{\xi}(\nabla v) \right) = 0 & \text{in }\Omega, \\[.2cm]
v=1 &\text{in }\Omega^{c}.\\
v\to 0&\text{as } x\to\infty.
\end{array}
\right.
\end{equation}

\begin{prop}
\label{arm}
If $1<p<n$, then there exists a unique positive solution $v_{p}\in C^{1,\alpha}_{\textrm{loc}}(\overline\Omega)\cap C^\infty(\overline\Omega\setminus \{\nabla v_p=0\})$ of \eqref{p_change}. If $\mathcal W_{r}(x_0)\subset \Omega^{c}\subset \mathcal W_{s}(y_0)$, then
\begin{equation}
\label{th:confr}
\left(\frac{r}{F^{o}(x-x_{0})}\right)^{\frac{n-p}{p-1}}\le v_{p}(x)
\le \left(\frac{s}{F^{o}(x-y_{0})}\right)^{\frac{n-p}{p-1}},\quad\forall x\in \Omega \setminus \{y_{0}\}; 
\end{equation}
Moreover, $v_{p}$ verifies
\begin{equation}
\label{limgrad}
\lim_{|x| \to \infty} \displaystyle \frac{F(\nabla v_{p})}{v_{p}}=0.
\end{equation}
\end{prop}
\begin{proof}
The proof of existence, uniqueness, regularity, as well as \eqref{th:confr}, follow by nowadays standard arguments; we refer the reader to \cite[Theorem 3.3]{bc} for the general anisotropic case we consider.

Finally we prove \eqref{limgrad}. We argue as in \cite{m}. Let $\eta \in C^{\infty}_{0}(\Omega)$ be a suitable cut-off function. Taking $\psi= v_{p}\eta^{p}$ as test function in the weak formulation for \eqref{p_change} and using the H\"older inequality, it easily follows that:
\[
\int_{\Omega} \eta^{p}F(\nabla v_{p})^{p}dx \le p^{p}\int_{\Omega}v_{p}^{p}F(\nabla \eta)^{p} dx.
\]
By Harnack inequality (see for instance \cite{t}) we get
\[
r^{p-n}\int_{B_{r/4}(x_0)} F(\nabla v_{p})^{p}dx \le C(n,p) \inf_{B_{r/2}(x_0)} v_p^p,
\] 
where $C(n,p)$ is a positive constant depending on $n$ and $p$. By applying the result contained in \cite{dib}, we have
\[
\|F(\nabla v_{p})\|_{L^{\infty}(B_{r/8}(x_{0}))} \le \frac{C(n,p)}{r} \inf_{B_{r/2}(x_0)} v_p,
\] 
which implies \eqref{limgrad}, and the proof is completed.
\end{proof}

It is direct to see that $$u_p=(1-p)\log v_p\in C^{1,\alpha}(\overline{\Omega})$$ solves \eqref{p_prob} (we refer the reader, e.g., also to \cite{dpgzaa} for problems involving equations as in \eqref{p_prob}). Next we show the gradient estimate in Theorem \ref{theoremapprox}, which is based on the following Lemma.
\begin{lem}\label{sperb}
Let $1<p<n$ and $u_p\in C^{1,\alpha}_{\textrm{loc}}(\overline\Omega)$ be a solution to \eqref{p_prob}, then 
\begin{equation}
\label{maxgrad}
\sup_{\bar\Omega}F(\nabla u_p)=\sup_{\de\Omega} F(\nabla u_p).
\end{equation}
\end{lem}
\begin{proof}
We omit the subscript $p$ in $u_p$ in the proof. Let $\tau= \sup_{\de\Omega} F(\nabla u)$. We consider the following set 
\[
\Omega_{\beta}=\left\{x\in\Omega\colon F(\nabla u)> \beta\right\},
\] 
with $\beta>\tau\ge 0$. From \eqref{limgrad} we see $F(\nabla u)$ vanishes at infinity by \eqref{limgrad}, then $\Omega_{\beta}$ is a bounded, open set such that $\overline\Omega_{\beta}\cap \de\Omega=\emptyset$ and $F(\nabla u)=\beta$ on $\de\Omega_{\beta}$. 

In order to prove \eqref{maxgrad}, we will prove that $\Omega_{\beta}=\emptyset$.

Note that in $\overline{\O_\b}$, $\nabla u\neq 0$ and hence $u\in C^\infty(\overline{\O_\b})$.
By writing $$G(\xi)=\frac{1}{2}F^2(\xi),$$ the equation in \eqref{p_prob} becomes
\begin{align}\label{eq-G}
\dive\left( G^{\frac p 2-1}(\nabla u) G_\xi(\nabla u) \right) =G^{\frac p2}(\nabla u).
\end{align}
 Hereafter we will adopt the Einstein convention on the repeated indices, and use the notations
\begin{align*}
&G=G(\nabla u),\quad G_{i}=G_{\xi_{i}}(\nabla u),\quad G_{ij}=G_{\xi_{i}\xi_{j}}(\nabla u),\\&\quad u_{i}=u_{x_{i}},\quad u_{ij}=u_{x_{i}x_{j}}, \cdots
\end{align*}
Differentiating \eqref{eq-G} with respect to ${x_{i}}$, we get
\[
\de_{x_{k}}\left(\de_{x_{i}}[G^{\frac p2-1} G_{k}]\right)=\frac p2 G^{\frac p2 -1}G_{j}u_{ij},
\]
and
\[
\de_{x_{k}}\left(G_{i}\de_{x_{i}}[G^{\frac p2-1} G_{k}]\right)=\frac p2 G^{\frac p2 -1}G_{i}G_{j}u_{ij}+G_{il}u_{lk}\de_{x_{i}}[G^{\frac p2-1} G_{k}].
\]
Then
\begin{multline*}
\de_{x_{k}}\left(\left[\frac{p-2}{2}G^{\frac p2-2}G_{j} G_{k}+G^{\frac p2 -1}G_{kj}\right]G_{i}u_{ij}\right)=\\=\frac p2 G^{\frac p2 -1}G_{i}G_{j}u_{ij}+\left[\frac{p-2}{2}G^{\frac p2-2}G_{j} G_{k}+G^{\frac p2 -1}G_{kj}\right]G_{il}u_{lk}u_{ij};
\end{multline*}
hence
\begin{multline}
\label{eq3}
\de_{x_{k}}\left(\left[\frac{p-2}{2}G^{\frac p2-2}G_{i}G_{j} G_{k}+G^{\frac p2 -1}G_{i}G_{kj}\right]u_{ij}\right)=\\=\frac p2 G^{\frac p2 -1}G_{i}G_{j}u_{ij}+\left[\frac{p-2}{2}G^{\frac p2-2}G_{j} G_{k}G_{il}u_{lk}u_{ij}+G^{\frac p2 -1}G_{kj}G_{il}u_{lk}u_{ij}\right]=
\\=
\frac p2 G^{\frac p2 -1}G_{i}G_{j}u_{ij}+\frac{p-1}{2}G^{\frac p2-2}G_{j} G_{k}G_{il}u_{lk}u_{ij} +\\+
G^{\frac p2-2}\left[-\frac{1}{2}G_{j} G_{k}G_{il}u_{lk}u_{ij}+G\,G_{kj}G_{il}u_{lk}u_{ij}\right].
\end{multline}
The Kato type inequality (see \cite[Lemma 2.2]{wx-pac}) implies that
\begin{equation*}
G G_{il}G_{jk}u_{ij}u_{kl}\ge \frac{1}2 G_{il}G_j G_k u_{ij}u_{kl}.
\end{equation*}
Hence, from \eqref{eq3} we get
\begin{multline*}
\de_{x_{k}}\left(\left[\frac{p-2}{2}G^{\frac p2-2}G_{i}G_{j} G_{k}+G^{\frac p2 -1}G_{i}G_{kj}\right]u_{ij}\right)\ge \\ \ge
\frac p2 G^{\frac p2 -1}G_{i}G_{j}u_{ij}+\frac{p-1}{2}G^{\frac p2-2}G_{j} G_{k}G_{il}u_{lk}u_{ij}.
\end{multline*}
The above inequality can be read as
\begin{multline*}
\dive\left[ G^{\frac p2-1}\left(G_{\xi\xi} \nabla_x G +\frac{p-2}{2} \frac{(G_\xi\cdot \nabla_x G)}{G}G_\xi \right) \right]
-\frac{p}{2}G^{\frac p2-1}( G_\xi\cdot\nabla_x G) \ge \\ \ge \frac{p-1}{2}( G_{\xi\xi} \nabla_x G)\cdot\nabla_x G,
\end{multline*}
that is
\begin{equation}
\label{equat}
\dive \left[ G^{\frac p2-1} A \nabla_x G \right]-\frac{p}{2} G^{\frac p2-1} G_\xi \cdot \nabla_x G \ge \frac{p-1}{2} G^{\frac{p}{2}-2} \left(G_{\xi\xi}\nabla_x G\right) \cdot \nabla_x G \ge 0,
\end{equation}
where 
\[
A=G_{\xi\xi}+\frac{p-2}{2}\frac{G_\xi\otimes G_\xi}{G}
\]
is a uniformly positive definite matrix. Hence, the functional in the left-hand side of \eqref{equat} can be seen as a linear elliptic operator acting on $G(\nabla u)$, and by the maximum principle we have that
\[
\sup_{\Omega_{\beta}}G(\nabla u) \le \sup_{\de \Omega_{\beta}}G(\nabla u) = \frac {\beta^{2}}{2}.
\]
This implies that $\Omega_{\beta}=\left\{x\in\Omega\colon G(\nabla u)>\frac{\beta^{2}}{2}\right\}$ is empty, and the proof is completed.
\end{proof}

\begin{proof}[Proof of Theorem \ref{theoremapprox}]
We are remained to prove the boundary gradient estimate \eqref{estcurv2}. The global gradient estimate \eqref{estcurv1} follows from Lemma \ref{sperb} and \eqref{estcurv2}.

Let $x\in \de \Omega$ such that $\mathcal W_r(x_0) \subset \Omega^c$. 
Since $u_p=0$ on $\de \O$, hence if $\nabla u_p(x)\neq 0$, then $\nu_F(x)=F_\xi(\nabla u_p(x))$ and $F(\nabla u_p(x))=\frac{\de u_p}{\de \nu_F}(x)$.
On the other hand, since  $W_r(x_0)$ and $\de \O$  are tangent at $x$, we see $x-x_0= r\nu_F(x)$.
It follows that 
\begin{eqnarray*}
F(\nabla u_p(x))&=&\frac{\de u_p}{\de \nu_F}(x)=\lim_{t\to 0}\frac{u_p(x+t\nu_F)}{t}\\&\leq &(n-p) 
\lim_{t\to 0}\frac{\log F^0(x+t\nu_F-x_0)-\log r}{t}
\\&=&\frac{n-p}{r}.
\end{eqnarray*}
Thus if we define 
\begin{equation}
\label{R}
R:=\sup\{r>0 \colon \forall x \in \de \Omega, \exists \mathcal W_r(x_0) \subset \Omega^c \text{ such that } x \in \de \mathcal W_r(x_0)\}.
\end{equation}
then \begin{equation*}\|F(\nabla u_{p})\|_{L^{\infty}(\de \Omega)}\le \frac{n-p}{R}.\end{equation*}
It follows from Lemma \ref{sperb} that
\begin{equation}
\label{estapp}
\|F(\nabla u_{p})\|_{L^{\infty}(\overline\Omega)}\le \frac{n-p}{R}.
\end{equation}

 
Next prove the estimate \eqref{estcurv2}. We argue as in \cite{HI1,kn}.
Let  $\eps>0$. Choose $\bar w\in C^\infty(\overline{\O})$ such that
\begin{itemize}
\item [i)]$
\bar w=0 \text{ on }  \de \Omega$ and $\bar w>0$ in $\Omega$;
\item [ii)]$
H_F^+ < F(\nabla\bar w)\le H_F^++\eps \text{ on } \de \Omega.$
\end{itemize}
Denote
\begin{equation}
\label{Q_p}
\mathcal Q_p[\varphi]:=\dive\left( F^{p-1}(\nabla \varphi)F_{\xi}(\nabla \varphi) \right) - F(\nabla \varphi)^{p}.
\end{equation}

Since $F(\nabla\bar w) >0$ and $\bar w=0$ on $\de \O$, the anisotropic mean curvature of $\de \O$ is given by
\begin{equation*}
H_F(x)= \dive \left(F_\xi(\nabla \bar w)\right). \end{equation*}
Thus
\[
\mathcal Q_1[\bar w](x)= H_F(x)-F(\nabla \bar w(x)) <0 \hbox{ for }  x\in \de \O. 
\]
 Let $\delta >0$ and denote by $U_{\delta}$ the components of the set $\{0 \le \bar w < \delta \} $ containing $\de \Omega$. If we choose $\delta>0$ small enough, 
 we may have $F(\nabla\bar w) >0$ and $\mathcal Q_1[\bar w]<0$ in $U_\delta$.
 
 Define $w\in C^\infty(U_\delta)$ by \[
w=\displaystyle\frac{\bar w}{1-\frac{\bar w}{\delta}}.
\]
Then \[
\nabla w=\frac{\nabla \bar w}{(1-\frac{\bar w}{\delta})^2}.
\]
A simple computation gives
\[
\mathcal Q_1[ w]= \mathcal Q_1[\bar w]+ \left( 1-\frac{1}{\left( 1-\frac{\bar w}{\delta}\right)^2}\right) F(\nabla \bar w) <0  \quad \text{ on } U_{\delta}.
\]

By \eqref{estapp}, there exists a constant $C=C(\delta)>0$, such that $u_{p}\le C$ in $U_{\delta}$. Denoted by $\tilde{U}_{C+1} $  the component of the set $\{0 \le w \le C+1\}$ in $U_{\delta}$. Since $u_p=w=0$ on $\de \O$, we see 
\begin{equation}
\label{cfront}
u_{p} \le w  \quad  \text{on}\;  \de \tilde{U}_{C+1}.
\end{equation}

In order to compare $u_{p}$ and $w$ in $\overline{\tilde{U}}_{C+1}$, we compute 
\[
\mathcal Q_p[ w]= F^{p-1}(\nabla w) \displaystyle \left( \mathcal Q_1[ w]+ (p-1) \frac{ w_{ik}F_{\xi_i}(\nabla  w) F_{\xi_k}(\nabla  w)}{F(\nabla w)} \right) \text{ in }\tilde{U}_{C+1}.
\]
Since $ \mathcal Q_1[ w]<0$  in $\tilde{U}_{C+1}$, one may choose $p-1$ small enough, depending on $\|\bar w\|_{C^2(\tilde{U}_{C+1})}$, $\inf_{\tilde{U}_{C+1}} F(\nabla \bar w)$ and $C$, such that \[
\mathcal Q_p[ w]<0\hbox{ in } \tilde{U}_{C+1}.
\]
 Then since $\mathcal Q_p[u_p]=0>\mathcal Q_p[ w]$ in $\tilde{U}_{C+1}$ and \eqref{cfront}, by the comparison principle, we have \[
 u_{p} \le w\hbox{ on }\overline{\tilde{U}}_{C+1}.\] 
It follows that
\[
F(\nabla u_{p})= \displaystyle \frac{\de u_{p}}{\de \nu_F}   \le  \displaystyle \frac{\de w}{\de \nu_F}=\displaystyle \frac{\de \bar w}{\de \nu_F} =F(\nabla \bar w) \le H_F^++\eps \quad \text{on }\de \Omega.
\]
The proof of Theorem \ref{theoremapprox} is completed.
\end{proof}
Now we are ready to prove Theorem \ref{thm-existence}.
\begin{proof}[Proof of the Theorem \ref{thm-existence}]
Let $u_{p}$ be the solution of \eqref{p_prob} in Theorem \ref{theoremapprox}.
Then for any precompact set $K\subset \Omega$, $u_{p}$ is also a minimizer of the functional
\begin{equation}
\label{funcjp}
J^{p}_{w}(\varphi)=\int_{K}\left[\frac{1}{p}F(\nabla \varphi)^{p}+\varphi F(\nabla w)^{p}\right]dx,
\end{equation}
in the sense that
\begin{equation}
\label{defsol}
J^{p}_{u_{p}}(u_{p}) \le J^{p}_{u_{p}}(\varphi),\quad \forall \varphi\in W^{1,p}_{\textrm{loc}}(\Omega)\text{ such that }\varphi=u_{p}\text{ in }\Omega\setminus K. 
\end{equation}
Indeed, being $u_{p}-\varphi=0$ outside $K$, and using it a as test function for \eqref{p_prob}, we get that
\begin{align*}
\int_{K} F^{p}(\nabla u_{p})(u_{p}-\varphi)dx&=\int_{K} F^{p-1}(\nabla u_{p})F_{\xi}(\nabla u_{p})\cdot\nabla (\varphi-u_{p})dx
\\&\leq \frac{1}{p}\int_K( F(\nabla \varphi)^p- F(\nabla u_p)^p) dx.
\end{align*}
The inequality above follows from the convexity of $F(\xi)^p$.

On the other hand, from \eqref{th:confr},  we know $u_p$ has uniform upper bound on any compact set in $\O$.
Since we also have uniform global gradient estimate \eqref{estcurv1} for $\nabla u_{p}$, by Arzela-Ascoli's theorem, we get that, there exists a subsequence $p_{k}\to 1^{+}$ and $u\in C^{0,1}_{loc}(\overline{\O})$ such that
\begin{equation}
\label{strongconvunif}
u_{p_{k}}\to u\text{  uniformly in any compact sets of }\overline\Omega.
\end{equation}  
If we can show $u$ is a weak solution of  \eqref{prob1}, then
since the weak solution of \eqref{prob1} is unique, we will get $u_p\to u$ uniformly convergence  in any compact sets of $\overline\Omega$ as $p\to 1$.

Next we show that $u$ is a proper weak solution to problem \eqref{prob1}, in the sense of Definition \ref{defweak}.
To this aim, we need to prove that 
\begin{equation}
\label{strongconvgrad}
\left|\nabla u_{p_{k}}\right|^{p_{k}} \to |\nabla u| \quad \text{in }L_{\textit{loc}}^{1}(\Omega)\text{ for a subsequence }p_k\to 1^{+}.
\end{equation}
Indeed from \eqref{strongconvunif} and \eqref{strongconvgrad} we can pass to the limit in \eqref{funcjp} obtaining that
\[
J^{K}_{u_{p_{k}}}(\varphi,p_{k})  \to J^{K}_{u}(\varphi) \quad \text{ and } \quad J^{K}_{u_{p_{k}}}(u_{p_{k}},p_{k})  \to J^{K}_{u}(u)
\] 
In order to prove \eqref{strongconvgrad}, we argue as in \cite{m,HI1}.
Let $K \subset \Omega$ be a precompact set and consider  the following test function $\psi$ in \eqref{defsol} 
\[
\psi= \eta \varphi + (1-\eta)u_{p},
\]
where $\eta \in C^{\infty}(\Omega)$ is a cut-off function such that  $0 \le \eta \le 1$, $\eta\equiv 1  $ in $K$, and $\varphi\in C^{0,1}_{\textrm{loc}}(\Omega)$. Then we get
\begin{multline*}
\int_{\supp\eta} \left(\frac{F^{p}(\nabla u_{p})}{p} +\eta\left(u_{p}-\varphi\right)F^{p}(\nabla u_{p})\right)dx \le \\ \le
\frac{1}{p}\int_{\supp\eta} F^{p}\left(\nabla (\eta\varphi+(1-\eta)u_{p})\right)dx \le \\ \le \frac{1}{p}\int_{\supp\eta} 
\left[(\varphi-u_{p})F(\nabla\eta)+\eta F(\nabla \varphi)+(1-\eta)F(\nabla u_{p})\right]^{p}dx \le\\ \le 
 \frac{3^{p-1}}{p}\int_{\supp\eta} 
\left[|\varphi-u_{p}|^{p}F^{p}(\nabla\eta)+\eta^{p} F^{p}(\nabla \varphi)+(1-\eta)^{p}F^{p}(\nabla u_{p})\right]dx.
\end{multline*}
Choosing $\varphi=u$, and letting $p_k\to 1^{+}$, we obtain
\[
\limsup_{p_{k}\to 1^{+}} \int_{\Omega} \eta F(\nabla u_{p_{k}})^{p_{k}}dx \le \int_{\Omega} \eta F(\nabla u)\,dx,
\] 
which together with Fatou's Lemma gives \eqref{strongconvgrad}.

 The properness of $u$ follows directly from \eqref{th:confr}.
The estimate follows from that in Theorem \ref{theoremapprox}.
The proof of  Theorem \ref{thm-existence} is completed. 
\end{proof}

\end{document}